\documentclass[12pt]{article}
\usepackage{tikz, amsmath, amssymb, amsthm, pgfmath, pgffor, arrayjob, indentfirst, caption, subcaption, enumitem, floatrow, setspace}

\newtheorem{theorem}{Theorem}
\newtheorem{corollary}{Corollary}
\newlist{thmlist}{enumerate}{1}
\setlist[thmlist]{label=$\arabic{thmlisti}$.  ,noitemsep}

\definecolor{R}{HTML}{FF0000}
\definecolor{G}{HTML}{3DF53D}
\definecolor{B}{HTML}{2255EE}
\definecolor{Y}{HTML}{FFFF00}%
\colorlet{X}{black!30}
\colorlet{Z}{black!70}
\colorlet{Q}{black!90}%
\colorlet{O}{orange}%
\colorlet{W}{white}%

\newarray\Front
\newarray\Right
\newarray\Top
\normalindextrue

\newcommand{\cube}[1]{
    \dataheight=#1
    \foreach \i in {1,...,#1} {
        \foreach \j in {1,...,#1} {
            \checkFront(\i,\j)
            \draw[line join=round,line cap=round,ultra thick,fill=\cachedata]%
            (\i-1,#1-\j) -- (\i-1,#1+1-\j) -- (\i,#1+1-\j) -- (\i,#1-\j)  -- cycle;
        }
    }
    \foreach \i in {1,...,#1} {
        \foreach \j in {1,...,#1} {
            \checkRight(\i,\j)
            \draw[line join=round,line cap=round,ultra thick,fill=\cachedata]%
            (#1+\i/3-1/3,#1-\j+\i/3-1/3) -- (#1+\i/3-1/3,#1+1-\j+\i/3-1/3) -- (#1+\i/3,#1+1-\j+\i/3) -- (#1+\i/3,#1-\j+\i/3)  -- cycle;
        }
    }
    \foreach \i in {1,...,#1} {
        \foreach \j in {1,...,#1} {
            \checkTop(\i,\j)
            \draw[line join=round,line cap=round,ultra thick,fill=\cachedata]%
            (\i-1+#1/3-\j/3,#1+#1/3-\j/3) -- (\i-1+#1/3-\j/3+1/3,#1+#1/3-\j/3+1/3) -- (\i-1+#1/3-\j/3+1/3+1,#1+#1/3-\j/3+1/3) -- (\i-1+#1/3-\j/3+1,#1+#1/3-\j/3)  -- cycle;
        }
    }
}

\title{$n \times n \times n$ Rubik's Cubes and God's Number \\ \Large A Group Theoretical Analysis}
\author{Daniel Salkinder}
\date{}

\begin{document}

\maketitle

\begin{abstract}
    The Rubik's Cube is the most popular puzzle in the world. Two of its studied aspects are God's Number, the minimum number of turns necessary to solve any state, and the first law of cubology, a solvability criterion. We modify previous statements of the first law of cubology for $n \times n \times n$ Rubik's Cubes, and prove necessary and sufficient solvability conditions. We compute the order of the Rubik's Cube group and the number of distinct configurations of the $n \times n \times n$ Rubik's Cube. Finally, we derive a lower bound for God's Number using the group theoretical results and a counting argument.
\end{abstract}

\section{Introduction} \label{sec:introduction}

The Rubik's Cube is the most famous and best-selling puzzle of all time, challenging generations of solvers since its invention by Erno Rubik in 1974. The original design is a $3 \times 3 \times 3$ cube where every face can rotate around its center. To solve the cube, one must rotate the faces to rearrange the cube into its solved state: where all faces are monochromatic. In 1981, Peter Sebesteny invented an extension of the Rubik's Cube to a $4 \times 4 \times 4$ cube where all slices could rotate. Only a couple of years later, a $5 \times 5 \times 5$ cube was created. Today, there are many variations of the Rubik's Cube, and cubes up to $17 \times 17 \times 17$ are widely sold. Additionally, there are variants designed like other platonic solids (Pyraminx, Megaminx, etc.), cubes that rotate on axes different than those of the Rubik's Cube (Skewb, Square One, etc.), and many more. This paper will focus on the $n \times n \times n$ Rubik's Cube.

The Rubik's Cube has garnered significant attention from throughout the scientific community. Many mathematicians have studied the cube, giving group theoretical analyses and solutions for certain cases \cite{bonzio_loi_peruzzi_2017, bonzio_loi_peruzzi_2018, sikiric_2020}. Another well-studied facet of the Rubik's Cube is God's Number: the minimum number of face turns needed to solve the puzzle from any position \cite{rokicki_2014, rokicki_kociemba_davidson_dethridge_2014, kunkle_cooperman_2007, weed_2016, demaine_eisenstat_rudoy_2017}. Additionally, the Rubik's Cube can act as a physical model of chaotic behaviour, which has led to applications in cryptography \cite{diaconu_loukhaoukha_2013, volte_patarin_nachef_2013} as well as physics \cite{czech_larjo_rozali_2011, lai_shi_pu_zhang_zhang_yu_li_2020, lee_huang_2008}.

With many speed Rubik's Cube solving competitions around the world, the puzzle begs the question, "What is the minimum number of moves necessary to solve the cube from any configuration?". Such an optimal algorithm has been deemed impossible for humans to find, and hence has been labeled \textit{God's Algorithm}, with the number of moves in the worst case being labeled \textit{God's Number}. Computer scientists like Rokicki \cite{rokicki_kociemba_davidson_dethridge_2014} have used symmetry and sheer computational power to determine that this number is 26 for the original Rubik's Cube, but this approach requires too much time to be applicable to general cubes. Others like Demaine \cite{demaine_demaine_eisenstat_lubiw_winslow_2011} have found the growth rate of God's Number for $n \times n \times n$ cubes.

We aim to approach the problem of finding God's Number from a group theoretical perspective. To employ such a strategy, we first need to be able to count the number of configurations of an $n \times n \times n$ Rubik's Cube. We do so with a criterion for solving the question "under which conditions is a (scrambled) cube solvable?", the first law of cubology. First discussed by Bandelow \cite{bandelow_1982} for the original Rubik's Cube in 1982, the first law of cubology has been attempted to be generalized by Bonzio et al. \cite{bonzio_loi_peruzzi_2018} in 2018, but the result of that paper proves insufficient to generate only valid configurations. 

The present work presents a modified first law of cubology for the $n \times n \times n$ Rubik's Cube. This is then used to find the number of distinct valid configurations. Finally, a counting and pigeonhole principle argument establishes the lower bound of God's Number as $\Omega(n^2/\log n)$, which is in accordance with \cite{demaine_demaine_eisenstat_lubiw_winslow_2011}, and is generalizable to find better coefficients.

The paper is structured as follows: Section \ref{sec:definitions} provides definitions and mathematical notations used throughout the paper. Section \ref{sec:group} discusses the $n \times n \times n$ Rubik's Cube group. Subsections \ref{subsec:sconf_subsets} and \ref{subsec:gn_subgroups} prove the necessary and sufficient conditions of the first law of cubology by analysing chosen subgroups. Subsection \ref{subsec:g_order} finds the order of the $n \times n \times n$ Rubik's Cube group. Section \ref{sec:numbers} gives a lower bound for God's Number. The conclusion suggests potential future research.

\section{Definitions} \label{sec:definitions}

The $n \times n \times n$ Rubik's Cube is an extension of the original $3 \times 3 \times 3$ Rubik's Cube. It takes the form of an $n \times n \times n$ cube made out of $1 \times 1 \times 1$ pieces, and each $n \times n \times 1$ slice is free to rotate around its center. As per standard convention, the $1 \times 1 \times 1$ cubes along the outside of a Rubik's Cube are called \textit{cubies} and the colored $1 \times 1$ faces of the cubies are called \textit{stickers}. We proceed to categorize the cubies as shown in Figure \ref{fig:cubie_types}:

\begin{figure}[h]
    \centering
    \begin{subfigure}[b]{.4\textwidth}
        \centering
        \begin{tikzpicture}[scale=.5]
            \readarray{Front}{%
                R&G&G&X&G&G&R&%
                G&B&O&O&O&B&G&%
                G&O&B&O&B&O&G&%
                X&O&O&Z&O&O&X&%
                G&O&B&O&B&O&G&%
                G&B&O&O&O&B&G&%
                R&G&G&X&G&G&R&%
            }
            \readarray{Right}{%
                R&G&G&X&G&G&R&%
                G&B&O&O&O&B&G&%
                G&O&B&O&B&O&G&%
                X&O&O&Z&O&O&X&%
                G&O&B&O&B&O&G&%
                G&B&O&O&O&B&G&%
                R&G&G&X&G&G&R&%
            }
            \readarray{Top}{%
                R&G&G&X&G&G&R&%
                G&B&O&O&O&B&G&%
                G&O&B&O&B&O&G&%
                X&O&O&Z&O&O&X&%
                G&O&B&O&B&O&G&%
                G&B&O&O&O&B&G&%
                R&G&G&X&G&G&R&%
            }
            \cube{7}
        \end{tikzpicture}\\
        \caption{Cubie types for a $7 \times 7 \times 7$ Rubik's Cube}
        \label{fig:cubie_types_odd}
    \end{subfigure}
    \hspace{1 cm}
    \begin{subfigure}[b]{.4\textwidth}
        \centering
        \begin{tikzpicture}[scale=.5]
            \readarray{Front}{%
                R&G&G&G&G&R&%
                G&B&O&O&B&G&%
                G&O&B&B&O&G&%
                G&O&B&B&O&G&%
                G&B&O&O&B&G&%
                R&G&G&G&G&R&%
            }
            \readarray{Right}{%
                R&G&G&G&G&R&%
                G&B&O&O&B&G&%
                G&O&B&B&O&G&%
                G&O&B&B&O&G&%
                G&B&O&O&B&G&%
                R&G&G&G&G&R&%
            }
            \readarray{Top}{%
                R&G&G&G&G&R&%
                G&B&O&O&B&G&%
                G&O&B&B&O&G&%
                G&O&B&B&O&G&%
                G&B&O&O&B&G&%
                R&G&G&G&G&R&%
            }
            \cube{6}
        \end{tikzpicture}\\
        \caption{Cubie types for a $6 \times 6 \times 6$ Rubik's Cube}
        \label{fig:cubie_types_even}
    \end{subfigure}
    \caption{Cubie types for odd and even $n$: Corners (Red), Single Edges (Grey), Coupled Edges (Green), Center Corners (Blue), Center Edges (Orange), and Fixed Centers (Dark Grey)}
    \label{fig:cubie_types}
\end{figure}
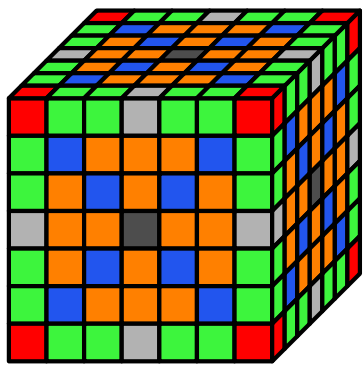
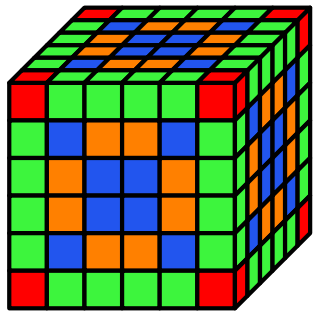

Cubies with 3 stickers are denoted \textit{corners}, with 2 stickers \textit{edges}, and with 1 sticker \textit{centers}. Centers can be split into \textit{fixed centers} (Dark Grey in Figure \ref{fig:cubie_types}), \textit{center corners} (Blue in Figure \ref{fig:cubie_types}), and \textit{center edges} (Orange in Figure \ref{fig:cubie_types}). Fixed centers are at the middle of their corresponding face and do not move when that face is rotated. We note that fixed centers only exist for odd cubes, as seen in Figure \ref{fig:cubie_types_odd} but not Figure \ref{fig:cubie_types_even}. Center corners are along the diagonals connecting the corner cubies of their corresponding face, and all remaining center cubies are center edges. Edges can be categorized as \textit{single edges} (Grey in Figure \ref{fig:cubie_types}) and \textit{coupled edges} (Green in Figure \ref{fig:cubie_types}). Single edges are at the middle of their corresponding edges and are in the same slice as fixed centers. Single edges also only exist for odd cubes, as seen in Figure \ref{fig:cubie_types_odd} but not Figure \ref{fig:cubie_types_even}. All remaining edges are called coupled due to the existence of two edges with the same sticker colors in each orbit of these edges. This paper will refer to the permutations of stickers on one cubie as \textit{orientations} of that cubie, and permutations of similar cubies \textit{permutations}.

We can also notate all possible moves on an $n \times n \times n$ Rubik's Cube. Like on a $3 \times 3 \times 3$ Rubik's Cube, the capital letters \textit{F, B, R, L, U, D} denote single clockwise $90^\circ$ turns of the Front, Back, Right, Left, Up, and Down faces, respectively. We call these moves \textit{face turns}. To label rotations of all possible slices, $F_k, B_k, R_k, L_k, U_k, D_k$, $1 \leq k \leq \left\lfloor\frac{n}{2}\right\rfloor$ denote single clockwise $90^\circ$ rotations of the $k^\text{th}$ slice from the corresponding face, as shown in Figure \ref{fig:move_notation}. For convenience, we omit the $1$ when writing \textit{F, B, R, L, U, D}. When $k>1$, we call these moves \textit{internal slice rotations}. Any sequence of these slice rotations is a valid move. \textit{$m^{-1}$} denotes the inverse of the move \textit{m} (a single counterclockwise $90^\circ$ rotation for slice rotations $m$). \textit{$m^2$} denotes doing move $m$ twice (a $180^\circ$ rotation of the slice moved by $m$ for slice rotations). We call the sequence of moves $m n m^{-1} n^{-1}$ a commutator, denoted $[m,n]$. Similarly, we call the sequence of moves $m n m^{-1}$ a conjugate, denoted $[m:n]$.

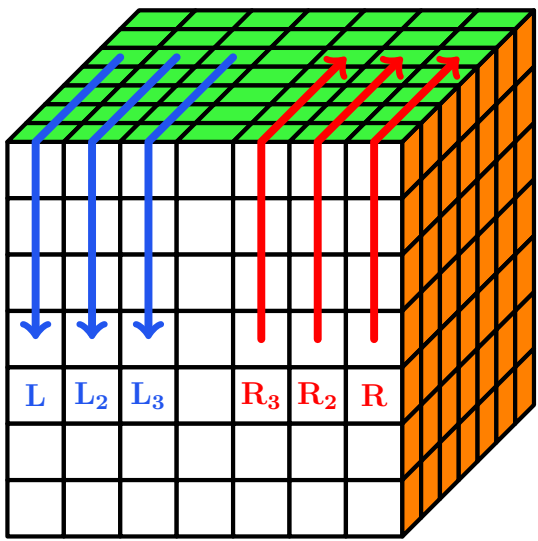
\begin{figure}[h]
    \centering
    \begin{tikzpicture}[scale=.75]
        \readarray{Front}{%
            W&W&W&W&W&W&W&%
            W&W&W&W&W&W&W&%
            W&W&W&W&W&W&W&%
            W&W&W&W&W&W&W&%
            W&W&W&W&W&W&W&%
            W&W&W&W&W&W&W&%
            W&W&W&W&W&W&W&%
        }
        \readarray{Right}{%
            O&O&O&O&O&O&O&%
            O&O&O&O&O&O&O&%
            O&O&O&O&O&O&O&%
            O&O&O&O&O&O&O&%
            O&O&O&O&O&O&O&%
            O&O&O&O&O&O&O&%
            O&O&O&O&O&O&O&%
        }
        \readarray{Top}{%
            G&G&G&G&G&G&G&%
            G&G&G&G&G&G&G&%
            G&G&G&G&G&G&G&%
            G&G&G&G&G&G&G&%
            G&G&G&G&G&G&G&%
            G&G&G&G&G&G&G&%
            G&G&G&G&G&G&G&%
        }
        \cube{7}
        \draw[->,line join=round,line cap=round,line width=3pt,color=R]%
        (6.5,3.5) -- (6.5, 7) -> (8, 8.5);
        \node [color=R] at (6.5,2.5) {$\mathbf{R}$};
        \draw[->,line join=round,line cap=round,line width=3pt,color=R]%
        (5.5,3.5) -- (5.5, 7) -> (7, 8.5);
        \node [color=R] at (5.5,2.5) {$\mathbf{R_2}$};
        \draw[->,line join=round,line cap=round,line width=3pt,color=R]%
        (4.5,3.5) -- (4.5, 7) -> (6, 8.5);
        \node [color=R] at (4.5,2.5) {$\mathbf{R_3}$};
        \draw[->,line join=round,line cap=round,line width=3pt,color=B]%
        (4, 8.5) -- (2.5, 7) -> (2.5,3.5);
        \node [color=B] at (2.5,2.5) {$\mathbf{L_3}$};
        \draw[->,line join=round,line cap=round,line width=3pt,color=B]%
        (3, 8.5) -- (1.5, 7) -> (1.5,3.5);
        \node [color=B] at (1.5,2.5) {$\mathbf{L_2}$};
        \draw[->,line join=round,line cap=round,line width=3pt,color=B]%
        (2, 8.5) -- (0.5, 7) -> (0.5,3.5);
        \node [color=B] at (0.5,2.5) {$\mathbf{L}$};
    \end{tikzpicture}\\
    \caption{Notation for slice rotations for a $7 \times 7 \times 7$ Rubik's Cube. Red marks the three slices corresponding to face R, and blue for face L}
    \label{fig:move_notation}
\end{figure}

For the following definitions, we proceed similarly to \cite{bonzio_loi_peruzzi_2018}. The set of all valid moves of the $n \times n \times n$ Rubik's Cube clearly forms a group with the operation of composition ($\cdot$), which we denote $\mathbf{M}_n$. This group is generated by the slice rotations defined above. We see that a move $m \in \mathbf{M}_n$ results in a permutation of the stickers of the cube. Since there are $6n^2$ stickers on an $n \times n \times n$ cube, we may define a group homomorphism
\begin{equation*}
    \varphi: \mathbf{M}_n \longrightarrow \mathbf{S}_{6n^2}
\end{equation*}
\noindent which sends moves $m$ to permutations $\varphi(m) \in \mathbf{S_{6n^2}}$ corresponding to the permutation of the stickers induced by $m$.

The symmetry group of the $n \times n \times n$ Rubik's Cube, which comprises of all permutations of stickers caused by a sequence of valid moves, can then be formally defined as $\mathbf{G} := \varphi(\mathbf{M}_n)$. This can also be expressed as $\mathbf{M}_n/ker(\varphi)$, which is formed by all classes of moves that produce the same configuration of the Rubik's Cube. For example, the moves $RL$ and $FB$ are different elements in $\mathbf{M}_n$ but both map to $\text{id}_{S_{6n^2}}$ under $\varphi$, meaning they correspond to the same element in $\mathbf{G}$.

Next, we consider two important subsets of $\mathbf{S}_{6n^2}$. We call the subset of $\mathbf{S}_{6n^2}$ corresponding to all permutations of stickers of the $n 
\times n \times n$ Rubik's Cube that can be obtained by disassembling and reassembling the cubies in such a way that preserves orbits of cubie permutations (i.e., we allow swapping two cubies if and only if there is a valid sequence of moves that brings one cubie to the position of the other cubie) the \textit{space of configurations} of the $n \times n \times n$ Rubik's Cube, denoted $\mathcal{S}_{\text{conf}}$ \footnote{This definition slightly differs from its analogue in \cite{bonzio_loi_peruzzi_2018} for purposes of generalizability to $n \times n \times n$ and clarification of allowed permutations}. We observe that $\mathcal{S}_{\text{conf}}$ can also be represented as the group generated by slice rotations, edge flips (which flip the two stickers on one edge cubie), and corner twists (which cycle the three stickers on one corner cubie). Clearly, $\mathbf{G} \subset \mathcal{S}_{\text{conf}}$. We call a configuration $s \in \mathcal{S}_{\text{conf}}$ valid if it can be obtained from the solved state by slice rotations. We call the subset of $\mathbf{S}_{6n^2}$ corresponding to all valid configurations that are physically distinct the \textit{space of distinct valid configurations}, denoted $\mathcal{S}_{\text{phys}}$.

\section{The $n \times n \times n$ Symmetry Group} \label{sec:group}

Let $\mathbf{G}$ act on the left on $\mathcal{S}_{\text{conf}}$:

\begin{gather*}
    \mathbf{G} \times \mathcal{S}_{\text{conf}} \longrightarrow \mathcal{S}_{\text{conf}}\\
    (g,s) \longmapsto g \cdot s
\end{gather*}

\subsection{The subgroups of $\mathcal{S}_{\text{conf}}$} \label{subsec:sconf_subsets}
In this section, we study the structure of $\mathbf{G}$ by analysing the subgroups of $\mathcal{S}_{\text{conf}}$. The significant subgroups of $\mathcal{S}_{\text{conf}}$ discussed in this section are $\mathbf{C'}$, which permutes and orients corner cubies, $\mathbf{E}_S'$, which permutes and orients single edges, $\mathbf{E}_{C_i}'$, which permutes and orients the orbit of coupled edges in slice $i$, $\mathbf{Z}_{C_i}'$, which permutes the orbit of center corners in slice $i$, and $\mathbf{Z}_{E_{i,j}}'$, which permutes the orbit of center edges in slices $i$ and $j$.

\begin{theorem} \label{thrm:corners_subgroup_all_perms}
    $\mathbf{C'} \cong \mathbf{S}_8 \times \mathbb{Z}_3^8$
\end{theorem}

\begin{theorem} \label{thrm:single_edges_subgroup_all_perms}
    $\mathbf{E}_S' \cong \mathbf{S}_{12} \times \mathbb{Z}_2^{12}$
\end{theorem}

\begin{figure}[b]
        \centering
        \begin{subfigure}[b]{.4\textwidth}
            \centering
            \begin{tikzpicture}[scale=.8]
                \readarray{Front}{%
                    Y&X&Y&%
                    X&W&X&%
                    Y&X&Y&%
                }
                \readarray{Right}{%
                    B&X&R&%
                    X&O&X&%
                    R&X&B&%
                }
                \readarray{Top}{%
                    R&X&B&%
                    X&G&X&%
                    B&X&R&%
                }
                \cube{3}
            \end{tikzpicture}\\
            \caption{Positions with $0$ orientation for one corner of a $3 \times 3 \times 3$ Rubiks's Cube}
            \label{fig:orientation_corners}
        \end{subfigure}
        \hspace{1 cm}
        \begin{subfigure}[b]{.4\textwidth}
            \centering
            \begin{tikzpicture}[scale=.8]
                \readarray{Front}{%
                    X&B&X&%
                    R&W&R&%
                    X&B&X&%
                }
                \readarray{Right}{%
                    X&R&X&%
                    B&O&B&%
                    X&R&X&%
                }
                \readarray{Top}{%
                    X&R&X&%
                    B&G&B&%
                    X&R&X&%
                }
                \cube{3}
            \end{tikzpicture}\\
            \caption{Positions with $0$ orientation for one single edge of a $3 \times 3 \times 3$ Rubiks's Cube}
            \label{fig:orientation_single_edges}
        \end{subfigure}
        \caption{Orientation definition for corner cubies and single edge cubies}
        \label{fig:orientation_corners_single}
    \end{figure}
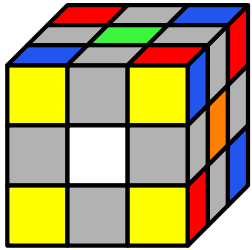
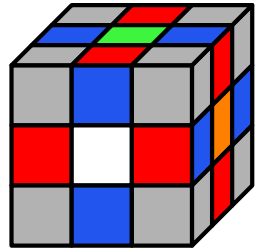

\begin{proof}
    First, we note that $\mathbf{C'}$ is the same for all cubes and that $\mathbf{E}_S'$ is the same for all odd cubes, since the only slice rotations that affect corners and single edges are outer face rotations. 
    
    We see that there are always 8 corners on any cube and 12 single edges on any odd cube, so we label permutations $\sigma \in \mathbf{S}_8$ for corners and $\tau_s \in \mathbf{S}_{12}$. 
    
    We proceed to define orientation for corners and center edges as shown in Figure \ref{fig:orientation_corners_single}. Orientation for a corner cubie is defined as $x_i \in \mathbb{Z}_3$ where 0 is when the White/Yellow sticker of the cubie lines up with the White/Yellow faces in the solved state (as seen in Figure \ref{fig:orientation_corners}), orientation 1 is a clockwise cycle of the stickers of orientation 0, and orientation 2 is a clockwise cycle of the stickers of orientation 1. We express $(x_1, x_2, \dots, x_8)$ as a vector $x \in \mathbb{Z}_3^8$.
    
    Orientation for a single edge cubie is defined as $z_i \in \mathbb{Z}_2$ where 0 is when the cubie matches the alternating pattern shown in Figure \ref{fig:orientation_single_edges} such that the orientation of the single edge in the solved state is 0. We express $(z_1, z_2, \dots, z_{12})$ as a vector $z \in \mathbb{Z}_2^{12}$.
    
    Therefore, there is a bijection between $\mathbf{C'}$ and pairs of permutations and orientation vectors $(\sigma, x)$, so $\mathbf{C'} \cong \mathbf{S}_8 \times \mathbb{Z}_3^8$. Similarly, there is a bijection between $\mathbf{E}_S'$ and pairs of permutations and orientation vectors $(\tau_s, z)$, so $\mathbf{E}_S' \cong \mathbf{S}_{12} \times \mathbb{Z}_2^{12}$.
\end{proof}

\begin{theorem} \label{thrm:coupled_edges_subgroup_all_perms}
    $\mathbf{E}_{C_i}' \cong \mathbf{S}_{24} \times \mathbb{Z}_2^{24}$
\end{theorem}

\begin{proof}
    We examine the orbit of one coupled edge cubie in slice $i$. We see that the cubie can be in one of 24 positions, so permutations for coupled edges can be labeled as $\tau_{c_i} \in \mathbf{S}_{24}, 1 < i \leq \left\lfloor\frac{n}{2}\right\rfloor$. One such orbit can be represented by looking at coupled edges on the $4 \times 4 \times 4$ as shown in Figure \ref{fig:orientation:coupled_edges}.
    
    A more interesting pattern emerges when we look at the orbit of the stickers on that cubie, as it does not reach all possible sticker locations on the orbit of the coupled edge cubie. Physically, this means that the orientation of coupled edges is fixed by their permutation. We define orientation for coupled edges by $y_{i_k} \in \mathbb{Z}_2$, where 0 orientation is matching the pattern in Figure \ref{fig:orientation:coupled_edges} such that the orientation of each coupled edge in the solved state is 0. Again, we express $(y_{i_1}, y_{i_2}, \dots, y_{i_{24}})$ as a vector $y_i \in \mathbb{Z}_2^{24}$.
    
    Therefore, similarly to our proof for corners and single edges, we establish a bijection between $\mathbf{E}_{C_i}'$ and the pair $(\tau_{c_i}, y_i)$, giving $\mathbf{E}_{C_i}' \cong \mathbf{S}_{24} \times \mathbb{Z}_2^{24}$.
\end{proof}

\begin{figure}[b]
    \centering
    \begin{tikzpicture}[scale=.8]
        \readarray{Front}{%
            X&G&O&X&%
            O&X&X&G&%
            G&X&X&O&%
            X&O&G&X&%
        }
        \readarray{Right}{%
            X&G&O&X&%
            O&X&X&G&%
            G&X&X&O&%
            X&O&G&X&%
        }
        \readarray{Top}{%
            X&G&O&X&%
            O&X&X&G&%
            G&X&X&O&%
            X&O&G&X&%
        }
        \cube{4}
    \end{tikzpicture}\\
    \caption{Positions with $0$ orientation for one coupled edge of a $4 \times 4 \times 4$ Rubiks's Cube}
    \label{fig:orientation:coupled_edges}
\end{figure}
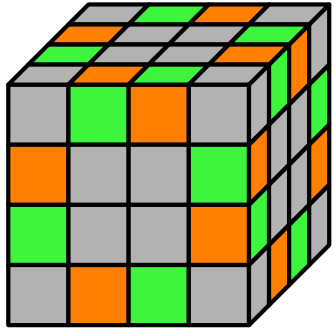

\begin{theorem} \label{thrm:center_corners_subgroup_all_perms}
    $\mathbf{Z}_{C_i}' \cong \mathbf{S}_{24}$
\end{theorem}

\begin{theorem} \label{thrm:center_edges_subgroup_all_perms}
    $\mathbf{Z}_{E_{i,j}}' \cong \mathbf{S}_{24}$
\end{theorem}

\begin{proof}
    We start by examining the orbits of one center cubie. Figure \ref{fig:orientation_centers_orbits} shows center cubie orbits for a $7 \times 7 \times 7$ cube. 
    
    We clearly see that center corners can be in any of 24 positions, so we label permutations of the centers in slice $i$ $\rho_{c_i} \in \mathbf{S}_{24}$ for $1 < i \leq \left\lfloor\frac{n}{2} \right\rfloor$. Since center cubies only have one sticker, they do not have orientation, so clearly $\mathbf{Z}_{C_i}' \cong \mathbf{S}_{24}$. 
    
    We note that orbits for center edges are also of size 24 (e.g. Red and Blue orbits in Figure \ref{fig:orientation_center_edges} never coincide). Labeling center edges provides more of a challenge, so label based on the slices $i$ and $j$ parallel to $U$ and $R$ respectively which contain the center edge in the orbit in the top left quadrant of the front face (i.e., in Figure \ref{fig:orientation_center_edges}, the red orbit would be labeled $(2,3)$, blue $(3,2)$, green $(2,4)$, and orange $(3,4)$). So, we can label permutations of center edges $\rho_{e_{i,j}} \in \mathbf{S}_{24}, 1 < i \leq \left\lfloor\frac{n}{2} \right\rfloor, 1 < j \neq i \leq \left\lceil\frac{n}{2} \right\rceil$. This also clearly gives $\mathbf{Z}_{E_{i,j}}' \cong \mathbf{S}_{24}$.
\end{proof}

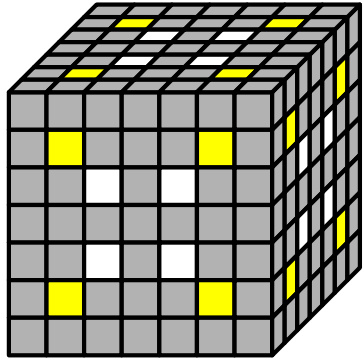
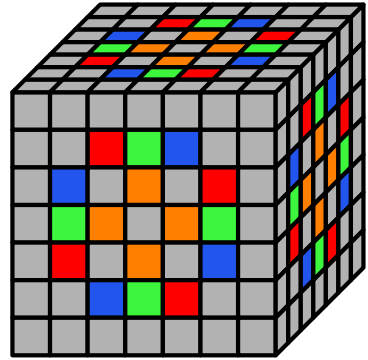
\begin{figure}[H]
        \centering
        \begin{subfigure}[b]{.4\textwidth}
            \centering
            \begin{tikzpicture}[scale=.5 ]
                \readarray{Front}{%
                    X&X&X&X&X&X&X&%
                    X&Y&X&X&X&Y&X&%
                    X&X&W&X&W&X&X&%
                    X&X&X&X&X&X&X&%
                    X&X&W&X&W&X&X&%
                    X&Y&X&X&X&Y&X&%
                    X&X&X&X&X&X&X&%
                }
                \readarray{Right}{%
                    X&X&X&X&X&X&X&%
                    X&Y&X&X&X&Y&X&%
                    X&X&W&X&W&X&X&%
                    X&X&X&X&X&X&X&%
                    X&X&W&X&W&X&X&%
                    X&Y&X&X&X&Y&X&%
                    X&X&X&X&X&X&X&%
                }
                \readarray{Top}{%
                    X&X&X&X&X&X&X&%
                    X&Y&X&X&X&Y&X&%
                    X&X&W&X&W&X&X&%
                    X&X&X&X&X&X&X&%
                    X&X&W&X&W&X&X&%
                    X&Y&X&X&X&Y&X&%
                    X&X&X&X&X&X&X&%
                }
                \cube{7}
            \end{tikzpicture}\\
            \caption{Orbits of center corners for a $7 \times 7 \times 7$ Rubiks's Cube}
            \label{fig:orientation_center_corners}
        \end{subfigure}
        \hspace{1 cm}
        \begin{subfigure}[b]{.4\textwidth}
            \centering
            \begin{tikzpicture}[scale=.5]
                \readarray{Front}{%
                    X&X&X&X&X&X&X&%
                    X&X&R&G&B&X&X&%
                    X&B&X&O&X&R&X&%
                    X&G&O&X&O&G&X&%
                    X&R&X&O&X&B&X&%
                    X&X&B&G&R&X&X&%
                    X&X&X&X&X&X&X&%
                }
                \readarray{Right}{%
                    X&X&X&X&X&X&X&%
                    X&X&R&G&B&X&X&%
                    X&B&X&O&X&R&X&%
                    X&G&O&X&O&G&X&%
                    X&R&X&O&X&B&X&%
                    X&X&B&G&R&X&X&%
                    X&X&X&X&X&X&X&%
                }
                \readarray{Top}{%
                    X&X&X&X&X&X&X&%
                    X&X&R&G&B&X&X&%
                    X&B&X&O&X&R&X&%
                    X&G&O&X&O&G&X&%
                    X&R&X&O&X&B&X&%
                    X&X&B&G&R&X&X&%
                    X&X&X&X&X&X&X&%
                }
                \cube{7}
            \end{tikzpicture}\\
            \caption{Orbits of center edges for a $7 \times 7 \times 7$ Rubiks's Cube}
            \label{fig:orientation_center_edges}
        \end{subfigure}
        \caption{Orbits of centers}
        \label{fig:orientation_centers_orbits}
    \end{figure}

\clearpage

\begin{theorem} \label{thrm:sconf_odd}
    $\mathcal{S}_{conf_{odd}} \cong \mathbf{C'} \times \mathbf{E}_S' \times \mathbf{E}_C'^{\frac{n-3}{2}} \times \mathbf{Z}_C'^{\frac{n-3}{2}} \times \mathbf{Z}_E'^{\frac{(n-3)^2}{4}}$
\end{theorem}

\begin{proof}
    $\mathcal{S}_{\text{conf}}$ is defined as the set of all permutations of orbits and orientation changes. Therefore, we establish a bijection between $\mathcal{S}_{\text{conf}}$ and the tuple $(\sigma, \tau_s, \tau_{c_i}, \rho_{c_i}, \rho_{e_{i,j}}, x, y_i, z)$, with all permutations and orientation vectors defined in the above proofs, $1 < i \leq \frac{n-1}{2}, 1 < j \neq i \leq \frac{n+1}{2}$. This tuple is equivalent to $((\sigma, x), (\tau_s,z), (\tau_{c_i},y_i), \rho_{c_i}, \rho_{e_{i,j}})$, which is equivalent to $(c, e_s, e_{c_i}, z_{c_i}, z_{e_{i,j}})$, $c \in \mathbf{C'}, e_s \in \mathbf{E}_S', e_{c_i} \in \mathbf{E}_{C_i}', z_{c_i} \in \mathbf{Z}_{C_i}'$, and $z_{e_{i,j}} \in \mathbf{Z}_{E_{i,j}}'$, so $\mathcal{S}_{conf_{odd}} \cong \mathbf{C'} \times \mathbf{E}_S' \times \mathbf{E}_{C_i}' \times \mathbf{Z}_{C_i}' \times \mathbf{Z}_{E_{i,j}}'$. Therefore, $\mathcal{S}_{conf_{odd}} \cong \mathbf{C'} \times \mathbf{E}_S' \times \mathbf{E}_C'^{\frac{n-3}{2}} \times \mathbf{Z}_C'^{\frac{n-3}{2}} \times \mathbf{Z}_E'^{\frac{(n-3)^2}{4}}$.
\end{proof}

\begin{theorem} \label{thrm:sconf_even}
    $\mathcal{S}_{conf_{even}} \cong \mathbf{C'} \times \mathbf{E}_C'^{\frac{n-2}{2}} \times \mathbf{Z}_C'^{\frac{n-2}{2}} \times \mathbf{Z}_E'^{\frac{(n-2)(n-4)}{4}}$
\end{theorem}

\begin{proof}
    Similarly to the proof of Theorem \ref{thrm:sconf_odd}, we establish a bijection between $\mathcal{S}_{\text{conf}}$ and $(c, e_{c_i}, z_{c_i}, z_{e_{i,j}})$, $c \in \mathbf{C'}, e_{c_i} \in \mathbf{E}_{C_i}', z_{c_i} \in \mathbf{Z}_{C_i}'$, and $z_{e_{i,j}} \in \mathbf{Z}_{E_{i,j}}'$, so $\mathcal{S}_{conf_{even}} \cong \mathbf{C'} \times \mathbf{E}_{C_i}' \times \mathbf{Z}_{C_i}' \times \mathbf{Z}_{E_{i,j}}' \cong \mathbf{C'} \times \mathbf{E}_C'^{\frac{n-2}{2}} \times \mathbf{Z}_C'^{\frac{n-2}{2}} \times \mathbf{Z}_E'^{\frac{(n-2)(n-4)}{4}}$.
\end{proof}

\begin{corollary} \label{crly:s_odd_card}
    $|\mathcal{S}_{conf_{odd}}| = 8! \cdot 3^8 \cdot 12! \cdot 2^{12(n-2)} \cdot (24!)^{\frac{(n-3)(n+1)}{4}}$
\end{corollary}

\begin{corollary} \label{crly:s_even_card}
    $|\mathcal{S}_{conf_{even}}| = 8! \cdot 3^8 \cdot 2^{12(n-2)} \cdot (24!)^{\frac{(n)(n-2)}{4}}$
\end{corollary}

We calculate $|\mathcal{S}_{\text{conf}}|$ for some small $n$. For a $2 \times 2 \times 2$ Rubik's Cube, we get $|\mathcal{S}_{\text{conf}}| = 8! \cdot 3^8 = |C'|$, which makes sense since the Pocket Cube is just made of 8 corners. The $3 \times 3 \times 3$ gives $|\mathcal{S}_{\text{conf}}| = 8! \cdot 3^8 \cdot 12! \cdot 2^{12} = |C' \times E_S'|$, which again makes sense since the original Rubik's Cube has only 8 corners and 12 edges. For a $5 \times 5 \times 5$, we get \footnote{This is in accordance with equation (1) from  \cite{bonzio_loi_peruzzi_2018}.} $|\mathcal{S}_{\text{conf}}| = 8! \cdot 3^8 \cdot 12! \cdot 2^{36} \cdot (24!)^3$.\\

We have shown that an element of $\mathcal{S}_{\text{conf}}$ can be represented as the tuple $(\sigma, \tau_s, \tau_{c_i}, \rho_{c_i}, \rho_{e_j}, x, y_i, z)$. Now, we use such tuples to represent elements of  $\mathcal{S}_{\text{conf}}$. The solved state or \textit{initial configuration} corresponds to $(\text{id}_{S_8}, \text{id}_{S_{12}}, \text{id}_{S_{24}\: i}, \text{id}_{S_{24}\: i}, \text{id}_{S_{24}\: i,j}, 0, 0, 0)$. \\

We now state and prove the necessary condition of the first law of cubology for an $n \times n \times n$ Rubik's Cube.\footnote{This statement of the first law of cubology refines the incorrect statement in Theorem 20 of  \cite{bonzio_loi_peruzzi_2018}. Condition 5 has been altered to match our definition of orientation in Theorem \ref{thrm:coupled_edges_subgroup_all_perms}. More importantly, condition 2 has been updated to give a stricter restraint on center edges.}

\begin{theorem} \label{thrm:first_law_necessary_odd}
    Any valid configuration $(\sigma, \tau_s, \tau_{c_i}, \rho_{c_i}, \rho_{e_{i,j}}, x, y_i, z)$ of an odd $n \times n \times n$ Rubik's Cube must satisfy the conditions:
    \begin{thmlist}
        \item $\text{sgn}(\sigma) = \text{sgn}(\tau_s) = \text{sgn}(\rho_{c_i})$
        \item $\text{sgn}(\rho_{e_{i,j}}) = \text{sgn}(\sigma)\text{sgn}(\tau_i)\text{sgn}(\tau_j)$
        \item $\sum_i x_i = 0$
        \item $\sum_i z_i = 0$
        \item $y_i = 0$
    \end{thmlist}
\end{theorem}

\begin{proof}
    We prove by checking that these conditions are invariant under the generators of $\mathbf{M}_n$. Since the solved state trivially satisfies all of these conditions, as any valid configurations are in the orbit of the solved state, this would show that all valid configurations satisfy these conditions.
    
    \begin{enumerate}
        \item Face turns consist of 4-cycles on corners, single edges, orbits of center corners, and orbits of center edges and two 4-cycles on orbits of coupled edges. Since this induces odd permutations on corners, single edges, and orbits of center corners, this preserves condition 1. 
        
        Internal slice rotations consist of 4-cycles on orbits of coupled edges and center edges, two 4-cycles on orbits of center corners, and act as the identity permutation on corners and single edges. Since this induces even permutations on corners, single edges, and orbits of center corners, this also preserves condition 1. 
        
        Therefore, $\text{sgn}(\sigma) = \text{sgn}(\tau_s) = \text{sgn}(\rho_{c_i})$.
        \item Face turns induce odd permutations on corners and any orbit of center edges and induce even permutations on any orbits of coupled edges. So, $\text{sgn}(\rho_{e_{i,j}}) = \text{sgn}(\sigma)\text{sgn}(\tau_i)\text{sgn}(\tau_j)$ holds under face turns. 
        
        Rotations of internal slice $i$ act as the identity on corners, induce odd permutations of coupled edges in orbit $i$ and also induce odd permutations on orbits of center edges $i,j$ or $j,i$. Therefore, $\text{sgn}(\rho_{e_{i,j}}) = \text{sgn}(\sigma)\text{sgn}(\tau_i)\text{sgn}(\tau_j)$ is preserved for each slice $i$ under internal slice rotations.
        \item Internal slice rotations act as the identity on corners, so we only consider the effects of face turns. By our definition of corner orientation in Theorem \ref{thrm:corners_subgroup_all_perms}, turning the white or yellow face does not change the orientation of any corners, and turning any other face increase two corner orientations by 1 and decreases two corner orientations by 1, which shows $\sum_i x_i = 0$.
        \item Internal slice rotations also act as the identity on single edges, so we again only consider the effects of face turns. By our definition of corner orientation in Theorem \ref{thrm:single_edges_subgroup_all_perms}, any face turn changes the orientation of all 4 single edges on that face, which preserves $\sum_i z_i = 0$.
        \item By the definition of orientation for coupled edges in Theorem \ref{thrm:coupled_edges_subgroup_all_perms}, the orientation of any coupled edge never changes due to it being fixed by the orbits of its stickers.
    \end{enumerate}
\end{proof}

\begin{theorem} \label{thrm:first_law_necessary_even}
    Any valid configuration $(\sigma, \tau_{c_i}, \rho_{c_i}, \rho_{e_{i,j}}, x, y_i, z)$ of an even $n \times n \times n$ Rubik's Cube must satisfy the conditions:
    \begin{thmlist}
        \item $\text{sgn}(\sigma) = \text{sgn}(\rho_{c_i})$
        \item $\text{sgn}(\rho_{e_{i,j}}) = \text{sgn}(\sigma)\text{sgn}(\tau_i)\text{sgn}(\tau_j)$
        \item $\sum_i x_i = 0$
        \item $y_i = 0$
    \end{thmlist}
\end{theorem}

\begin{proof}
    We proceed as for Theorem \ref{thrm:first_law_necessary_odd} and prove by checking that these conditions are invariant under the face turns and internal slice rotations
    
    \begin{enumerate}
        \item Face turns induce odd permutations on corners and any orbit of center corners, which preserves condition 1. Internal slice rotation act as the identity on corners and induce even permutations of the orbit of center corners in their slice, so condition 1 holds.
        \item Fact turns act as odd permutations on corners and center edges and as odd permutations on coupled edges, which preserved condition 2. Internal slice rotations induce odd permutations of center edges in their slice and coupled edges in their slice, while acting as the identity on other coupled edges and corners, which also preserves condition 2.
        \item This is equivalent to Condition 3 of Theorem \ref{thrm:first_law_necessary_odd}
        \item This is equivalent to Condition 5 of Theorem \ref{thrm:first_law_necessary_odd}
    \end{enumerate}
\end{proof}

\subsection{The subgroups of $\mathbf{G}$} \label{subsec:gn_subgroups}
In this section, we continue to study the structure of $\mathbf{G}$ by analyzing some important \textit{commuting} subgroups. For these subgroups, we again use a construction analogous to that of \cite{bonzio_loi_peruzzi_2018}. Namely, the first subgroups discussed in this section are $\mathbf{C}$, which permutes only corners, preserves orientation, and acts as the identity on all other cubies, $\mathbf{E}_S$, which permutes only single edges, $\mathbf{E}_{C_i}$, which permutes only coupled edges in slice $i$, $\mathbf{Z}_{C_i}$, which permutes only center corners in slice $i$, and $\mathbf{Z}_{E_{i,j}}$, which permutes only center edges in slice $i,j$. For each subgroup, we consider odd and even cubes separately. Then, we discuss the subgroups $\mathbf{O}_C$ and $\mathbf{O}_E$ which orient only corners and single edges respectively and preserve permutation.

\begin{theorem} \label{thrm:corners_only_subgroup}
    $\mathbf{C}_{odd} \cong \mathbf{C}_{even} \cong \mathcal{A}_8$
\end{theorem}

\begin{proof}
    We start as in our proof of Theorem \ref{thrm:corners_subgroup_all_perms} by noting that $\mathbf{C}$ is the same for any size Rubik's Cube.
    
    First, we show $\mathcal{A}_8 \leqslant \mathbf{C}$. We see that the move
    \begin{equation}
        m_c = [[R:U],D] = (RUR^{-1})D(RU^{-1}R^{-1})D^{-1}
    \end{equation}
    
    is a 3-cycle on corners for a $3 \times 3 \times 3$ Rubik's Cube (as shown in Figure \ref{fig:corners_3_cycle}) and for any $n \times n \times n$ Rubik's Cube. Since all the corners are in the same orbit, for any there exists an element $g \in G$ that brings it to one of the cubies permuted by $m_c$, so $g \cdot m_c \cdot g^{-1}$ gives 3-cycles containing any corners, which suffice to generate $\mathcal{A}_8$.
    
    Since there are 8 corners, clearly $\mathbf{C} \leqslant \mathbf{S}_8$. By condition 1 of Theorem \ref{thrm:first_law_necessary_odd} and condition 1 of Theorem \ref{thrm:first_law_necessary_even}, a permutation of corners that does not permute center corners must be even for both odd and even cubes, so $\mathbf{C} \leqslant \mathcal{A}_8$, which gives $\mathbf{C}_{odd} \cong \mathbf{C}_{even} \cong \mathcal{A}_8$.
\end{proof}

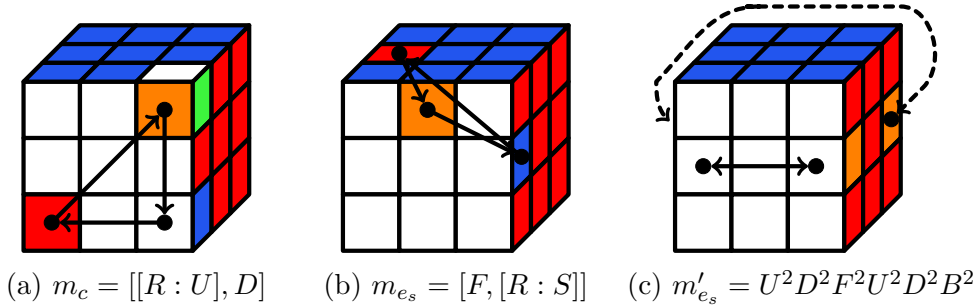
\begin{figure}[H]
    \centering
    \begin{subfigure}[b]{.3\textwidth}
        \centering
        \begin{tikzpicture}[scale=.75]
            \readarray{Front}{%
                W&W&O&%
                W&W&W&%
                R&W&W&%
            }
            \readarray{Right}{%
                G&R&R&%
                R&R&R&%
                B&R&R&%
            }
            \readarray{Top}{%
                B&B&B&%
                B&B&B&%
                B&B&W&%
            }
            \cube{3}
            \node[draw, shape=circle, fill=black, minimum size=2mm, inner sep=0] (a) at (.5,.5) {};
            \node[draw, shape=circle, fill=black, minimum size=2mm, inner sep=0] (b) at (2.5,2.5) {};
            \node[draw, shape=circle, fill=black, minimum size=2mm, inner sep=0] (c) at (2.5,.5) {};
            \draw[->, line join=round,line cap=round,ultra thick] (a) -> (b);
            \draw[->, line join=round,line cap=round,ultra thick] (b) -> (c);
            \draw[->, line join=round,line cap=round,ultra thick] (c) -> (a);
        \end{tikzpicture}\\
        \caption{$m_c = [[R:U],D]$}
        \label{fig:corners_3_cycle}
    \end{subfigure}
    \begin{subfigure}[b]{.3\textwidth}
        \centering
        \begin{tikzpicture}[scale=.75]
            \readarray{Front}{%
                W&O&W&%
                W&W&W&%
                W&W&W&%
            }
            \readarray{Right}{%
                R&R&R&%
                B&R&R&%
                R&R&R&%
            }
            \readarray{Top}{%
                B&B&B&%
                R&B&B&%
                B&B&B&%
            }
            \cube{3}
            \node[draw, shape=circle, fill=black, minimum size=2mm, inner sep=0] (a) at (1.5,2.5) {};
            \node[draw, shape=circle, fill=black, minimum size=2mm, inner sep=0] (b) at (3+1/6,1.5+1/6) {};
            \node[draw, shape=circle, fill=black, minimum size=2mm, inner sep=0] (c) at (.5+1/3+1/6,3+1/3+1/6) {};
            \draw[->, line join=round,line cap=round,ultra thick] (a) -> (b);
            \draw[->, line join=round,line cap=round,ultra thick] (b) -> (c);
            \draw[->, line join=round,line cap=round,ultra thick] (c) -> (a);
        \end{tikzpicture}\\
        \caption{$m_{e_s} = [F,[R:S]]$}
        \label{fig:single_edges_3_cycle}
    \end{subfigure}
    \begin{subfigure}[b]{.35\textwidth}
        \centering
        \begin{tikzpicture}[scale=.75]
            \readarray{Front}{%
                W&W&W&%
                W&W&W&%
                W&W&W&%
            }
            \readarray{Right}{%
                R&R&R&%
                O&R&O&%
                R&R&R&%
            }
            \readarray{Top}{%
                B&B&B&%
                B&B&B&%
                B&B&B&%
            }
            \cube{3}
            \node[draw, shape=circle, fill=black, minimum size=2mm, inner sep=0] (a) at (2.5,1.5) {};
            \node[draw, shape=circle, fill=black, minimum size=2mm, inner sep=0] (b) at (.5,1.5) {};
            \node[draw, shape=circle, fill=black, minimum size=2mm, inner sep=0] (c) at (3+5/6,2+1/3) {};
            \draw[<->, line join=round,line cap=round,ultra thick] (a) -> (b);
            \draw[<->, dashed, line join=round,line cap=round,ultra thick] (c) to[out=45, in=-135] (4.4,2.9) to[out=45, in=0] (4,4+1/3) to[out=180, in=180] (1.5,4+1/3) to[out=180,in=45] (-1/3,3) to[out=-90,in=135] (-0.1,2.3);
        \end{tikzpicture}\\
        \caption{$m_{e_s}' = U^2D^2F^2U^2D^2B^2$}
        \label{fig:single_edges_dual_pair}
    \end{subfigure}
    \caption{3-cycles for corners and single edges}
    \label{fig:corners_single_edges_3_cycle}
\end{figure}

\begin{theorem} \label{thrm:single_edges_only_subgroup}
    $\mathbf{E}_S \cong \mathcal{A}_{12}$
\end{theorem}

\begin{proof}
    We proceed as we did for Theorem \ref{thrm:corners_only_subgroup}. First, we show $\mathcal{A}_{12} \leqslant \mathbf{E}_S$. We want to create a move that permutes elements in $\mathbf{E}_S$. 
    
    This is simplified if we allow slice rotations of the central slice. We can reconcile the issue of centers being fixed in our definition of the cube by considering such a move as a combination of all parallel slice rotations and then a shift in notation for following moves. As such, we denote the move $M$ as the rotation of the central slice between $L$ and $R$, $E$ between $U$ and $D$, and $S$ between $F$ and $B$. Figure \ref{fig:single_edges_dual_pair} shows a move equivalent to $M^2B^2M^2B^2$. 
    We see that the move
    \begin{equation}
        m_{e_s} = [F,[R:S]]
    \end{equation}
    
    is a 3-cycle on single edges for a $3 \times 3 \times 3$ Rubik's Cube (as shown in Figure \ref{fig:single_edges_3_cycle}) that holds for any $n \times n \times n$ Rubik's Cube. Since all 12 single edges are in the same orbit, for any there exists an element $g \in G$ that brings it to one of the cubies permuted by $m_{e_s}$, so $g \cdot m_{e_s} \cdot g^{-1}$ gives 3-cycles containing any single edges, which suffice to generate $\mathcal{A}_{12}$.
    
    Since there are 12 single edges, clearly $\mathbf{E}_S \leqslant \mathbf{S}_{12}$. By condition 1 of Theorem \ref{thrm:first_law_necessary_odd}, a permutation of single edges that does not permute corners must be even, so $\mathbf{E}_S \leqslant \mathcal{A}_{12}$, which gives $\mathbf{E}_S \cong \mathcal{A}_{12}$.
\end{proof}

\begin{theorem} \label{thrm:center_corners_only_subgroup}
    $\mathbf{Z}_{C_i \: odd} \cong \mathbf{Z}_{C_i \: even} \cong \mathcal{A}_{24}$
\end{theorem}

\begin{theorem} \label{thrm:center_edges_only_subgroup}
    $\mathbf{Z}_{E_{i,j} \: odd} \cong \mathbf{Z}_{E_{i,j} \: even} \cong \mathcal{A}_{24}$
\end{theorem}

\begin{proof}
    We proceed as we did for Theorem \ref{thrm:corners_only_subgroup}. First, we show $\mathcal{A}_{24} \leqslant \mathbf{Z}_{C_i}$ and $\mathcal{A}_{24} \leqslant \mathbf{Z}_{E_{i,j}}$. We see that the move
    \begin{equation}
        m_{z_{c_2}} = [[R_2^{-1},D_2],F^{-1}]
    \end{equation}
    
    is a 3-cycle on center corners in the second slice of an $n \times n \times n$ Rubik's Cube, as shown in Figure \ref{fig:center_corners_3_cycle}. We generalize this to move
    \begin{equation}
        m_{z_{i,j}} = [[R_j^{-1},D_i],F^{-1}]
    \end{equation}
    
    which gives a 3-cycle on centers in slice $i$ of an $n \times n \times n$ Rubik's Cube. We note that when $i = j$, this is a 3-cycle on center corners and when $i \neq j$ on center edges, as shown in Figure \ref{fig:center_edges_3_cycle}. Also, we substitute $M$ for $R_{\frac{n+1}{2}}$ to give a 3-cycle on center edges whose orbit is in the central slice\footnote{This move with center slice turns can serve as a replacement for the incorrect move $w$ on the $5 \times 5 \times 5$ Rubik's Cube in Theorem 13 of \cite{bonzio_loi_peruzzi_2018}, which permuted center corners instead of center edges}. Again, by a move $g \cdot m_{z_{i,j}} \cdot g^{-1}$, we 3-cycle centers in the same orbit, giving $\mathcal{A}_{24} \leqslant \mathbf{Z}_{C_i}$ and $\mathcal{A}_{24} \leqslant \mathbf{Z}_{E_{i,j}}$.
    
    Since there are 24 centers in any orbit, clearly $\mathbf{Z}_{C_i} \leqslant \mathbf{S}_{24}$ and $\mathbf{Z}_{E_{i,j}} \leqslant \mathbf{S}_{24}$. 
    
    By condition 1 of Theorem \ref{thrm:first_law_necessary_odd} and condition 1 of Theorem \ref{thrm:first_law_necessary_even}, a permutation of center corners that does not permute corners must be even, so $\mathbf{Z}_{C_i} \leqslant \mathcal{A}_{24}$, which gives $\mathbf{Z}_{C_i \: odd} \cong \mathbf{Z}_{C_i \: even} \cong \mathcal{A}_{24}$.  
    
    By condition 2 of Theorem \ref{thrm:first_law_necessary_odd} and condition 2 of Theorem \ref{thrm:first_law_necessary_even}, a permutation of center edges that does not permute corners or coupled edges must be even, which means $\mathbf{Z}_{E_{i,j}} \leqslant \mathcal{A}_{24}$, so $\mathbf{Z}_{E_{i,j} \: odd} \cong \mathbf{Z}_{E_{i,j} \: even} \cong \mathcal{A}_{24}$\footnote{This corrects the subgroup of center edges for a $6 \times 6 \times 6$ Rubik's Cube from Theorem 17 of \cite{bonzio_loi_peruzzi_2018}}.
\end{proof}

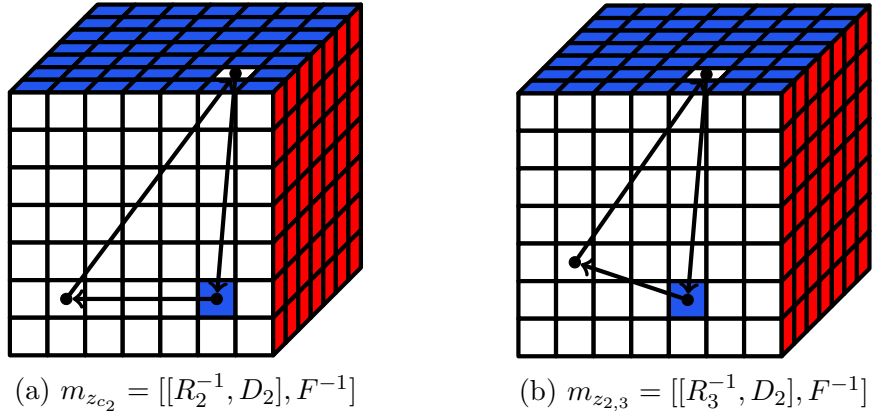
\begin{figure}[t]
    \centering
    \begin{subfigure}[b]{.4\textwidth}
        \centering
        \begin{tikzpicture}[scale=.5]
            \readarray{Front}{%
                W&W&W&W&W&W&W&%
                W&W&W&W&W&W&W&%
                W&W&W&W&W&W&W&%
                W&W&W&W&W&W&W&%
                W&W&W&W&W&W&W&%
                W&W&W&W&W&B&W&%
                W&W&W&W&W&W&W&%
            }
            \readarray{Right}{%
                R&R&R&R&R&R&R&%
                R&R&R&R&R&R&R&%
                R&R&R&R&R&R&R&%
                R&R&R&R&R&R&R&%
                R&R&R&R&R&R&R&%
                R&R&R&R&R&R&R&%
                R&R&R&R&R&R&R&%
            }
            \readarray{Top}{%
                B&B&B&B&B&B&B&%
                B&B&B&B&B&B&B&%
                B&B&B&B&B&B&B&%
                B&B&B&B&B&B&B&%
                B&B&B&B&B&B&B&%
                B&B&B&B&B&W&B&%
                B&B&B&B&B&B&B&%
            }
            \cube{7}
            \node[draw, shape=circle, fill=black, minimum size=1.5mm, inner sep=0] (a) at (1.5,1.5) {};
            \node[draw, shape=circle, fill=black, minimum size=1.5mm, inner sep=0] (b) at (5.5+1/3+1/6,7+1/3+1/6) {};
            \node[draw, shape=circle, fill=black, minimum size=1.5mm, inner sep=0] (c) at (5.5,1.5) {};
            \draw[->, line join=round,line cap=round,ultra thick] (a) -> (b);
            \draw[->, line join=round,line cap=round,ultra thick] (b) -> (c);
            \draw[->, line join=round,line cap=round,ultra thick] (c) -> (a);
        \end{tikzpicture}\\
        \caption{$m_{z_{c_2}} = [[R_2^{-1},D_2],F^{-1}]$}
        \label{fig:center_corners_3_cycle}
    \end{subfigure}
    \hspace{1 cm}
    \begin{subfigure}[b]{.4 \textwidth}
        \centering
        \begin{tikzpicture}[scale=.5]
            \readarray{Front}{%
                W&W&W&W&W&W&W&%
                W&W&W&W&W&W&W&%
                W&W&W&W&W&W&W&%
                W&W&W&W&W&W&W&%
                W&W&W&W&W&W&W&%
                W&W&W&W&B&W&W&%
                W&W&W&W&W&W&W&%
            }
            \readarray{Right}{%
                R&R&R&R&R&R&R&%
                R&R&R&R&R&R&R&%
                R&R&R&R&R&R&R&%
                R&R&R&R&R&R&R&%
                R&R&R&R&R&R&R&%
                R&R&R&R&R&R&R&%
                R&R&R&R&R&R&R&%
            }
            \readarray{Top}{%
                B&B&B&B&B&B&B&%
                B&B&B&B&B&B&B&%
                B&B&B&B&B&B&B&%
                B&B&B&B&B&B&B&%
                B&B&B&B&B&B&B&%
                B&B&B&B&W&B&B&%
                B&B&B&B&B&B&B&%
            }
            \cube{7}
            \node[draw, shape=circle, fill=black, minimum size=1.5mm, inner sep=0] (a) at (4.5,1.5) {};
            \node[draw, shape=circle, fill=black, minimum size=1.5mm, inner sep=0] (b) at (1.5,2.5) {};
            \node[draw, shape=circle, fill=black, minimum size=1.5mm, inner sep=0] (c) at (5,7.5) {};
            \draw[->, line join=round,line cap=round,ultra thick] (a) -> (b);
            \draw[->, line join=round,line cap=round,ultra thick] (b) -> (c);
            \draw[->, line join=round,line cap=round,ultra thick] (c) -> (a);
        \end{tikzpicture}\\
        \caption{$m_{z_{2,3}} = [[R_3^{-1},D_2],F^{-1}]$}
        \label{fig:center_edges_3_cycle}
    \end{subfigure}
    \caption{3-cycles for centers}
    \label{fig:centers_3_cycle}
\end{figure}

\begin{theorem} \label{thrm:coupled_edges_only_subgroup}
    $\mathbf{E}_{C_i \: odd} \cong \mathcal{A}_{24}$, $\mathbf{E}_{C_i \: even} \cong \mathcal{S}_{24}$
\end{theorem}

\begin{proof}
    Once again, we first show $\mathcal{A}_{24} \leqslant \mathbf{E}_{C_i}$. We see that the move
    \begin{equation}
        m_{e_{c_2}} = [[F^{-1},U],D_2]
    \end{equation}
    
    \begin{figure}[ht]
        \centering
        \begin{subfigure}[b]{.4\textwidth}
            \centering
            \begin{tikzpicture}[scale=.8]
                \readarray{Front}{%
                    W&W&W&W&%
                    W&W&W&W&%
                    R&W&W&R&%
                    W&W&W&W&%
                }
                \readarray{Right}{%
                    R&R&O&R&%
                    R&R&R&R&%
                    B&R&R&R&%
                    R&R&R&R&%
                }
                \readarray{Top}{%
                    B&B&B&B&%
                    B&B&B&W&%
                    B&B&B&B&%
                    B&B&B&B&%
                }
                \cube{4}
                \node[draw, shape=circle, fill=black, minimum size=2mm, inner sep=0] (a) at (.5,1.5) {};
                \node[draw, shape=circle, fill=black, minimum size=2mm, inner sep=0] (b) at (4+5/6,3.5+5/6) {};
                \node[draw, shape=circle, fill=black, minimum size=2mm, inner sep=0] (c) at (3.5,1.5) {};
                \draw[->, line join=round,line cap=round,ultra thick] (a) -> (b);
                \draw[->, line join=round,line cap=round,ultra thick] (b) -> (c);
                \draw[->, line join=round,line cap=round,ultra thick] (c) -> (a);
            \end{tikzpicture}\\
            \caption{$m = [[F^{-1},U],D_2]$}
            \label{fig:coupled_edges_3_cycle}
        \end{subfigure}
        \hspace{1 cm}
        \begin{subfigure}[b]{.5 \textwidth}
            \centering
            \begin{tikzpicture}[scale=.8]
                \readarray{Front}{%
                    W&B&B&W&%
                    W&W&W&W&%
                    W&W&W&W&%
                    W&W&W&W&%
                }
                \readarray{Right}{%
                    R&R&R&R&%
                    R&R&R&R&%
                    R&R&R&R&%
                    R&R&R&R&%
                }
                \readarray{Top}{%
                    B&B&B&B&%
                    B&B&B&B&%
                    B&B&B&B&%
                    B&W&W&B&%
                }
                \cube{4}
            \end{tikzpicture}\\
            \caption{An odd permutation in $\mathbf{E}_{C_i \: even}$}
            \label{fig:odd_perm_for_coupled_edges}
        \end{subfigure}
        \caption{Selected permutations in $\mathbf{E}_{C_i}$}
        \label{fig:coupled_edges_cycles}
    \end{figure}
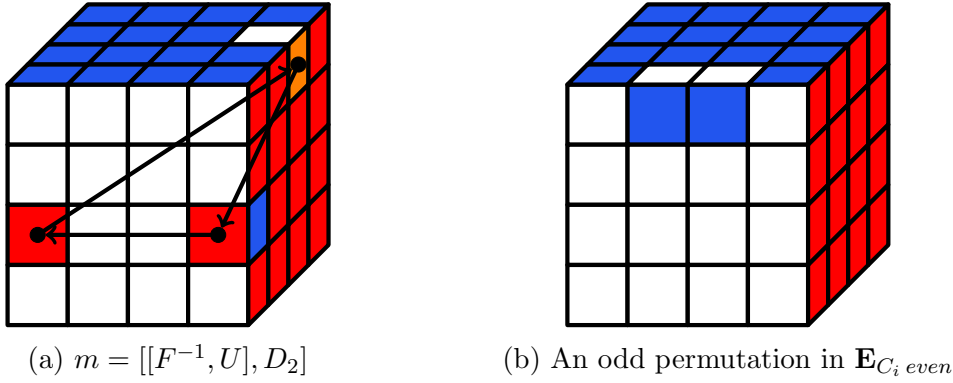
    
    is a 3-cycle on coupled edges in the second slice of an $n \times n \times n$ Rubik's Cube, as shown in Figure \ref{fig:coupled_edges_3_cycle}. We generalize this to move
    \begin{equation}
        m_{e_{c_i}} = [[F^{-1},U],D_i]
    \end{equation}
    
    which gives a 3-cycle on centers in slice $i$ of an $n \times n \times n$ Rubik's Cube. By a move $g \cdot m_{e_{c_i}} \cdot g^{-1}$, we 3-cycle coupled edges in the same orbit, giving $\mathcal{A}_{24} \leqslant \mathbf{E}_{C_i}$.
    
    Since there are 24 coupled edges in any orbit, clearly $\mathbf{E}_{C_i} \leqslant \mathbf{S}_{24}$. 
    
    For odd cubes, by condition 2 of Theorem \ref{thrm:first_law_necessary_odd}, a permutation of coupled edges that does not permute corners or center edges must be even, so $\mathbf{E}_{C_i \: odd} \leqslant \mathcal{A}_{24}$, which gives $\mathbf{E}_{C_i \: odd} \cong \mathcal{A}_{24}$. 
    
    For even cubes, we can find a move that induces an odd permutation of center edges. For a $4 \times 4 \times 4$ or $5 \times 5 \times 5$ Rubik's Cube, one such move is $m=(R^2R_2^2B^2,U^2L_2U^2R_2^{-1}U^2R_2U^2F^2R_2F^2L_2^{-1})$ (shown in Figure \ref{fig:odd_perm_for_coupled_edges}), which can be generalized to create an odd permutation for orbit $i$ of an $n \times n \times n$ Rubik's Cube by substituting $R^2R_2^2\dots R_i^2$ for $R^2R_2^2$, $R_i$ for $R_2$, and $L_i$ for $L_2$ and then combining such moves. Therefore, from this odd permutation and $\mathcal{A}_{24}$, we can generate all odd permutations, so $\mathbf{E}_{C_i \: even} \cong \mathcal{S}_{24}$.
\end{proof}

\begin{theorem} \label{thrm:corners_orient_subgroup}
    $\mathbf{O}_{C} \cong \mathbb{Z}_{3}^7$
\end{theorem}

\begin{proof}
    First, we show $\mathbb{Z}_{3}^7 \leqslant \mathbf{O}_{C}$
    We see that the move
    \begin{equation}
        m_{o_c} = [[F,L']^2,U]
    \end{equation}
    
    changes the orientation of 2 corners for any $n \times n \times n$ Rubik's Cube, as shown in Figure \ref{fig:corner_orientation_swap}. Again, for some $g \in G$, $g \cdot m_c \cdot g^{-1}$ can change the orientation of 2 corners $\mod 3$, which generates $\mathbb{Z}_{3}^7$.
    
    Since there are 8 corners, clearly $\mathbf{O}_{C} \leqslant \mathbb{Z}_{3}^8$. By condition 3 of Theorem \ref{thrm:first_law_necessary_odd} and Theorem \ref{thrm:first_law_necessary_even}, $\mathbf{O}_{C} < \mathbb{Z}_{3}^8$, so $\mathbb{Z}_{3}^7 \leqslant \mathbf{O}_{C}$
\end{proof}

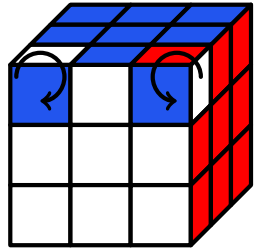
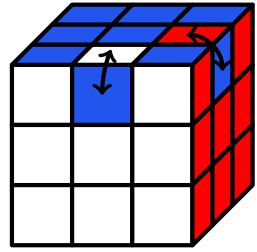
\begin{figure}[ht]
    \centering
    \begin{subfigure}[b]{.3\textwidth}
        \centering
        \begin{tikzpicture}[scale=.8]
            \readarray{Front}{%
                B&W&B&%
                W&W&W&%
                W&W&W&%
            }
            \readarray{Right}{%
                W&R&R&%
                R&R&R&%
                R&R&R&%
            }
            \readarray{Top}{%
                B&B&B&%
                B&B&B&%
                W&B&R&%
            }
            \cube{3}
            \draw[->, line join=round,line cap=round,ultra thick] (0.1,2.8) arc (180:-90:.4);
            \draw[->, line join=round,line cap=round,ultra thick] (3+1/6,2.8) arc (0:270:.4);
        \end{tikzpicture}\\
        \caption{$m_{o_c} = [[F,L']^2,U]$}
        \label{fig:corner_orientation_swap}
    \end{subfigure}
    \hspace{1 cm}
    \begin{subfigure}[b]{.4\textwidth}
        \centering
        \begin{tikzpicture}[scale=.8]
            \readarray{Front}{%
                W&B&W&%
                W&W&W&%
                W&W&W&%
            }
            \readarray{Right}{%
                R&B&R&%
                R&R&R&%
                R&R&R&%
            }
            \readarray{Top}{%
                B&B&B&%
                B&B&R&%
                B&W&B&%
            }
            \cube{3}
            \draw[<->, line join=round,line cap=round,ultra thick] (1.5,2.5) arc (180:158:2);
            \draw[<->, line join=round,line cap=round,ultra thick] (2.9,3.5) arc (90:0:0.6);
        \end{tikzpicture}\\
        \caption{$m_{o_e} = [FEF^2E^2F,U]$}
        \label{fig:edge_orientation_swap}
    \end{subfigure}
    \caption{Orientation swaps for corners and single edges}
    \label{fig:orientation_swaps}
\end{figure}

\begin{theorem} \label{thrm:edges_orient_subgroup}
    $\mathbf{O}_{E} \cong \mathbb{Z}_{2}^{11}$
\end{theorem}

\begin{proof}
    We show $\mathbb{Z}_{2}^{11} \leqslant \mathbf{O}_{C}$
    We see that the move
    \begin{equation}
        m_{o_e} = [FEF^2E^2F,U]
    \end{equation}
    
    changes the orientation of 2 single edges, as shown in Figure \ref{fig:edge_orientation_swap}. For some $g \in G$, $g \cdot m_c \cdot g^{-1}$ can change the orientation of 2 single edges $\mod 2$, which generates $\mathbb{Z}_{2}^{11}$.
    
    Since there are 12 corners, clearly $\mathbf{O}_{C} \leqslant \mathbb{Z}_{2}^{11}$. By condition 4 of Theorem \ref{thrm:first_law_necessary_odd}, $\mathbf{O}_{C} < \mathbb{Z}_{2}^{12}$, so $\mathbf{O}_{C} \cong \mathbb{Z}_{2}^{11}$
\end{proof}

Now, we can prove the sufficient condition of the first law of cubology:

\begin{theorem} \label{thrm:first_law_sufficient_odd}
    Any configuration $(\sigma, \tau_s, \tau_{c_i}, \rho_{c_i}, \rho_{e_{i,j}}, x, y_i, z)$ of an odd $n \times n \times n$ is valid if:
    \begin{thmlist}
        \item $\text{sgn}(\sigma) = \text{sgn}(\tau_s) = \text{sgn}(\rho_{c_i})$
        \item $\text{sgn}(\rho_{e_{i,j}}) = \text{sgn}(\sigma)\text{sgn}(\tau_i)\text{sgn}(\tau_j)$
        \item $\sum_i x_i = 0$
        \item $\sum_i z_i = 0$
        \item $y_i = 0$
    \end{thmlist}
\end{theorem}

\begin{proof}
    We by prove by showing there exists a sequence of moves that brings a configuration $(\sigma, \tau_s, \tau_{c_i}, \rho_{c_i}, \rho_{e_{i,j}}, x, y_i, z)$ satisfying the above conditions to the solved state, $(\text{id}_{S_8}, \text{id}_{S_{12}}, \text{id}_{S_{24}\: i}, \text{id}_{S_{24}\: i}, \text{id}_{S_{24}\: i,j}, 0, 0, 0)$. 
    
    First, we undergo a series of moves that will make all our permutations even. If $\text{sgn}(\sigma) = -1$, we do $F$. Now, we know $\text{sgn}(\sigma) = \text{sgn}(\tau_s) = \text{sgn}(\rho_{c_i}) = +1$. If $\text{sgn}(\tau_i) = -1$, we do $F_i$. Now, we know $\text{sgn}(\rho_{e_{i,j}}) = \text{sgn}(\sigma)\text{sgn}(\tau_i)\text{sgn}(\tau_j) = +1$. We call this move sequence $e$. We let
    \begin{equation*}
        e \cdot (\sigma, \tau_s, \tau_{c_i}, \rho_{c_i}, \rho_{e_{i,j}}, x, y_i, z) = (\sigma', \tau_s', \tau_{c_i}', \rho_{c_i}', \rho_{e_{i,j}}', x', y_i', z')
    \end{equation*}
    Next, we get the permutations to the solved state. By Theorem \ref{thrm:corners_only_subgroup}, there exists a move $p_c \in \mathbf{C}$ such that 
    \begin{equation*}
        p_c \cdot (\sigma', \tau_s', \tau_{c_i}', \rho_{c_i}', \rho_{e_{i,j}}', x', y_i', z') = (\text{id}_{S_8}, \tau_s', \tau_{c_i}', \rho_{c_i}', \rho_{e_{i,j}}', x', y_i', z')
    \end{equation*}
    Similarly, by Theorems \ref{thrm:single_edges_only_subgroup} and \ref{thrm:center_corners_only_subgroup}, there exist moves $p_{e_s} \in \mathbf{E}_S$, and $p_{z_{c_i}} \in \mathbf{Z}_{C_i}$ such that
    \begin{equation*}
        p_{z_{c_i}} \cdot p_{e_s} \cdot (\text{id}_{S_8}, \tau_s', \tau_{c_i}', \rho_{c_i}', \rho_{e_{i,j}}', x', y_i', z') = (\text{id}_{S_8}, \text{id}_{S_{12}}', \tau_{c_i}', \text{id}_{S_{24}\: i}, \rho_{e_{i,j}}', x', y_i', z')
    \end{equation*}
    And, by Theorems \ref{thrm:coupled_edges_only_subgroup} and \ref{thrm:center_edges_only_subgroup}, there exist moves $p_{e_{c_i}} \in \mathbf{E}_{C_i}$, and $p_{z_{e_{i,j}}} \in \mathbf{Z}_{E_{i,j}}$ such that
    \begin{equation*}
        p_{z_{e_i}} \cdot p_{e_{c_i}} \cdot (\text{id}_{S_8}, \text{id}_{S_{12}}', \tau_{c_i}', \text{id}_{S_{24}\: i}, \rho_{e_{i,j}}', x', y_i', z') = (\text{id}_{S_8}, \text{id}_{S_{12}}', \text{id}_{S_{24}\: i}, \text{id}_{S_{24}\: i}, \text{id}_{S_{24}\: i,j}, x', y_i', z')
    \end{equation*}
    
    Then, we can deal with orientation. By condition 5 of Theorem \ref{thrm:first_law_necessary_odd}, $y_i' = 0$. By Theorems \ref{thrm:corners_orient_subgroup} and \ref{thrm:edges_orient_subgroup} and conditions 3 and 4 of \ref{thrm:first_law_necessary_odd}, there exist moves $o_c \in \mathbf{O}_C$ and $o_e \in \mathbf{O}_E$ such that 
    \begin{equation*}
        o_e \cdot o_c \cdot (\text{id}_{S_8}, \text{id}_{S_{12}}', \text{id}_{S_{24}\: i}, \text{id}_{S_{24}\: i}, \text{id}_{S_{24}\: i,j}, x', y_i', z') = (\text{id}_{S_8}, \text{id}_{S_{12}}', \text{id}_{S_{24}\: i}, \text{id}_{S_{24}\: i}, \text{id}_{S_{24}\: i,j}, 0, 0, 0)
    \end{equation*}
    
    Therefore, the move $m = ep_cp_{e_{c_i}}p_{z_{c_i}}p_{e_s}p_{z_{s_i}}o_co_e \in G$ brings any configuration that satisfies the above conditions to the solved state, meaning any configuration that satisfies the conditions is valid because it can be reached from the solved state by move $m^{-1}$ 
\end{proof}

\begin{theorem} \label{thrm:first_law_sufficient_even}
    Any configuration $(\sigma, \tau_{c_i}, \rho_{c_i}, \rho_{e_{i,j}}, x, y_i, z)$ of an even $n \times n \times n$ is valid if:
    \begin{thmlist}
        \item $\text{sgn}(\sigma) = \text{sgn}(\rho_{c_i})$
        \item $\text{sgn}(\rho_{e_{i,j}}) = \text{sgn}(\sigma)\text{sgn}(\tau_i)\text{sgn}(\tau_j)$
        \item $\sum_i x_i = 0$
        \item $y_i = 0$
    \end{thmlist}
\end{theorem}

\begin{proof}
    By similar logic to the proof of Theorem \ref{thrm:first_law_sufficient_even}, there exist moves $e$, $p_c \in \mathbf{C}$, $p_{z_{c_i}} \in \mathbf{Z}_{C_i}$, $p_{e_{c_i}} \in \mathbf{E}_{C_i}$, $p_{z_{e_{i,j}}} \in \mathbf{Z}_{E_{i,j}}$, $o_c \in \mathbf{O}_C$, and $o_e \in \mathbf{O}_E$ such that $m = ep_cp_{e_{c_i}}p_{z_{c_i}}p_{e_s}p_{z_{s_i}}o_co_e \in G$ brings any configuration that satisfies the above conditions to the solved state, meaning any such configuration is valid.
\end{proof}

\subsection{The order of $\mathbf{G}$} \label{subsec:g_order}

We that note the action of $\mathbf{G}$ on $\mathcal{S}_{\text{conf}}$ is free, i.e. $\forall s \in \mathcal{S}_{\text{conf}}, g \cdot s = s \Rightarrow g = \text{id}_\mathbf{G}$ and hence all stabilizers $G_s = {g \in \mathbf{G}: g\cdot s = s}$ are clearly trivial. Therefore, there is also a bijection between any orbit $\mathcal{O} = \mathbf{G} \cdot s = {g \cdot s | g \in \mathbf{G}}$ and $\mathbf{G}$. Therefore, by the orbit counting theorem (or the Cauchy-Frobenius-Burside lemma), the number of orbits $|\mathcal{S}_{\text{conf}} \backslash \mathbf{G}|$ is

\begin{equation*}
    |\mathcal{S}_{\text{conf}} \backslash \mathbf{G}| = \frac{1}{|\mathbf{G}|}\sum_{s \in \mathcal{S}_{\text{conf}}}|\mathbf{G}_s| = \frac{|\mathcal{S}_{\text{conf}}|}{|\mathbf{G}|}
\end{equation*}

So, we can find the order of group $\mathbf{G}$ by counting the number of orbits in $\mathcal{S}_{\text{conf}}$. We can count by analysing the formulations of the law of cubology with different initial configurations. 

First, we count the number of orbits for odd cubes. For the first condition, we can assign $+1$ or $-1$ to each of $\text{sgn}(\sigma)$, $\text{sgn}(\tau)_s$, and $sgn{\rho_{c_i}}$, which will only be in the same orbit if all signs are flipped, giving $2^{2+\frac{n-3}{2}}/2=2^{\frac{n-1}{2}}$ choices. For the second condition, we have choices $\text{sgn}(\rho_{e_{i,j}}) = \pm \text{sgn}(\sigma)\text{sgn}(\tau_i)\text{sgn}(\tau_j)$ for each center edge,  giving $2^{\frac{(n-3)^2}{4}}$ choices. For condition 3, we can have $\sum_i x_i = 0$, $1$, or $2$, giving $3$ choices. Similarly, condition 4 gives $2$ choices. We can choose any starting $y_i$, which we recall is a vector containing coupled edge orientations in orbit $i$, giving us $2^{24 \cdot \frac{n-3}{2}}$ choices. So, $|\mathcal{S}_{\text{conf}} \backslash \mathbf{G}| = 2^{\frac{(n-3)^2}{4}+\frac{n+1}{2}} \cdot 3 \cdot 2^{12(n-3)}$.

For even cubes, the first condition gives $2^{\frac{n-2}{2}}$ choices. The second condition gives $2^{\frac{(n-2)(n-4)}{4}}$, the third $3$, and the last $2^{12(n-2)}$. Therefore, $|\mathcal{S}_{\text{conf}} \backslash \mathbf{G}| = 2^{\frac{(n-2)^2}{4}} \cdot 3 \cdot 2^{12(n-2)}$.

\begin{corollary}
    $|\mathbf{G}|_{odd} = 8! \cdot 3^7 \cdot 12! \cdot 2^{11} \cdot (24!)^{\frac{(n-3)(n+1)}{4}} / 2^{(\frac{n-3}{2})^2+(\frac{n-3}{2})+1}$
\end{corollary}

\begin{corollary}
    $|\mathbf{G}|_{even} = 8! \cdot 3^7 \cdot (24!)^{\frac{(n)(n-2)}{4}} / 2^{(\frac{n-2}{2})^2}$
\end{corollary}

We can compute some values to compare against established results. For the $2 \times 2 \times 2$, we get $|\mathbf{G}| = 8! \cdot 3^7$, which is equal to $|\mathbf{C}| \cdot |\mathbf{O}_C|$. For the $3 \times 3 \times 3$ Rubik's Cube, $|\mathbf{G}| = 8! \cdot 3^7 \cdot 12! \cdot 2^{11}$, which gives the well known $43 quintillion$ valid configurations. For the $5 \times 5 \times 5$, we get $|\mathbf{G}| = 8! \cdot 3^7 \cdot 12! \cdot 2^8 \cdot (24!)^3$\footnote{Once again, this rectifies the incorrect order for $\mathbf{G}_5$ in  \cite{bonzio_loi_peruzzi_2018}}. For all $2 \leq n \leq 9$, this formula matches the output of a self-made GAP program that generates and finds the order of the symmetry group of the $n \times n \times n$ Rubik's Cube (the program takes too long to be viable for larger $n$).

\section{God's Number} \label{sec:numbers}
We recall that God's Number is defined as the minimum number of face turns needed to solve the puzzle from any position. It has been shown to be 26 for the original Rubik's Cube (we use the quarter turn metric). This section derives a lower bound for God's Number by considering the least number of moves necessary to reach at least the number of physically distinct positions. So, we want to calculate the size of $\mathcal{S}_{\text{phys}}$. 

We let $G$ act on the left on $\mathcal{S}_{\text{phys}}$:

\begin{gather*}
    \mathbf{G} \times \mathcal{S}_{\text{phys}} \longrightarrow \mathcal{S}_{\text{phys}}\\
    (g,s) \longmapsto g \cdot s
\end{gather*}

By the Cauchy-Frobenius-Burnside lemma,

\begin{equation*}
    |\mathcal{S}_{\text{phys}} \backslash \mathbf{G}| = \frac{1}{|\mathbf{G}|}\sum_{s \in \mathcal{S}_{\text{phys}}}|\mathbf{G}_s|
\end{equation*}

Since $\mathcal{S}_{\text{phys}} \subset \mathbf{G}$, $|\mathcal{S}_{\text{phys}} \backslash \mathbf{G}| = 1$, since for every pair $s_1,s_2 \in \mathcal{S}_{\text{phys}}$, there exists $g \in \mathbf{G}$ such that $g \cdot s_1 = s_2$.

\begin{theorem}
    $|\mathbf{G}_s|_{odd} = \left(\frac{24^6}{2}\right)^{\frac{(n-3)(n-1)}{4}}$, $|\mathbf{G}_s|_{even} = \left(\frac{24^6}{2}\right)^{\frac{(n-2)^2}{4}}$
\end{theorem}

$\mathbf{G}_s$ is the stabilizer of $s$ in $G$, so to count $|\mathbf{G}_s|$, we must determine how many elements $g \in G$ do not change the physical appearance of the cube, thus all mapping to the same physically distinct configuration. Clearly, any permutation or orientation change of corners or edges produces a physically distinct configuration, so we only count valid permutations of centers that appear the same.

For any orbit of centers, permutations of the $4$ centers on each face result in physically identical valid configurations. We also recall that by conditions 1 and 2 of Theorems \ref{thrm:first_law_necessary_odd} and \ref{thrm:first_law_necessary_even}, a permutation within an orbit of center corners or center edges that does not permute corners or edges must be even. Therefore, we get $4!^6/2$ such permutations for each orbit. The result clearly follows from the number of orbits of centers for odd and even cubes.\\\\

\begin{corollary}
    $|\mathcal{S}_{\text{phys}}|_{odd} = 8! \cdot 3^7 \cdot 12! \cdot 2^{10} \cdot (24!)^{\frac{(n-3)(n+1)}{4}} / (24^6)^{\frac{(n-3)(n-1)}{4}}$
\end{corollary}

\begin{corollary}
    $|\mathcal{S}_{\text{phys}}|_{even} = 8! \cdot 3^7 \cdot (24!)^{\frac{(n)(n-2)}{4}} / (24^6)^{\frac{(n-2)^2}{4}}$
\end{corollary}

\begin{theorem}
    God's Number grows as $\Omega(n^2/\log n)$
\end{theorem}

\begin{proof}
    We have $3n$ slices (counting center slices for odd cubes) for an $n \times n \times n$ Rubik's Cube, which means we have $6n$ basic moves. After $k$ moves, the cube can achieve at most $(6n)^k$ configurations. By the pigeonhole principle, if any valid configuration can be solved in $k$ basic moves, the $\sum (6n)^k$ configurations reached in $k$ moves must be more than the number of distinct valid configurations. Therefore, for even cubes,
    
    \begin{gather*}
        8! \cdot 3^7 \cdot (24!)^{\frac{(n)(n-2)}{4}} / (24^6)^{\frac{(n-2)^2}{4}} \leq \sum (6n)^k \leq (6n)^{k+1}\\
        (k+1) \log 6n \geq \frac{(n)(n-2)}{4} \log 24! - \frac{(n-2)^2}{4} \log 24^6\\
        k \geq \frac{\frac{(n)(n-2)}{4} \log 24! - \frac{(n-2)^2}{4} \log 24^6}{\log 6n}
    \end{gather*}
\end{proof}

We note that this argument can be tuned by finding the number of sequences of $k$ moves that give distinct configurations (for example, omitting all instances of $mm^{-1}$ or $m^3$) to give a better coefficient for the lower bound.

\section{Conclusion} \label{sec:conclusion}

We have modified and proven the law of cubology, which gives a necessary and sufficient solvability criterion for the $n \times n \times n$ Rubik's Cube. We derived the order of the Rubik's Cube group and number of distinct valid configurations for any size cube. We were able to determine an $\Omega(n^2/\log n)$ lower bound for God's Number.

These results can be applied to cryptography and physics by quantifying the number of states when the randomized Rubik's Cube is used as a physical model for chaos.

However, there are still many more pertinent questions in this field. The following are a few suggestions:

In addition to God's Number, The Devil's Number examines the shortest length of an algorithm that passes every possible state. An analysis of the diameter of the Rubik's Cube group could provide a lower bound for this less widely studied quantity.

Randomization of a Rubik's Cube is a fascinating topic. God's Number could be applied to determine the optimal number of turns to make to scramble a cube.

The Rubik's Cube can be further generalized to other shapes, and even more dimensions. A group theoretical approach can give the number of configurations of such constructions.

\clearpage
\nocite{GAP4}
\nocite{MIT_2010}
\bibliographystyle{abbrv}
\bibliography{citations}

\end{document}